\documentclass[10pt]{article}
\usepackage{graphicx} 
\usepackage{epstopdf}
\usepackage{amsmath}
\usepackage{amsfonts}
\usepackage{amssymb}
\usepackage{amsmath}
\usepackage{amsthm}
\usepackage{latexsym}
\usepackage{color}
\usepackage[english]{babel}
\usepackage{float}
\usepackage{caption}
\usepackage{subfigure}
\usepackage{tikz}
\usetikzlibrary{matrix,arrows,decorations.pathmorphing}

\newtheorem{theorem}{Theorem}[section]

\newtheorem{corollary}[theorem]{Corollary}
\newtheorem{definition}[theorem]{Definition}

\newtheorem{lemma}[theorem]{Lemma}

\newtheorem{proposition}[theorem]{Proposition}

\begin{document}

\title{Rabinowitz-Floer Homology for Super-quadratic Dirac Equations On Compact spin Manifolds}
\author{Ali Maalaoui$^{(1)}$}
\addtocounter{footnote}{1}
\footnotetext{Department of Mathematics,
Rutgers University - Hill Center for the Mathematical Sciences
110 Frelinghuysen Rd., Piscataway 08854-8019 NJ, USA. E-mail address:
{\tt{maalaoui@math.rutgers.edu}}}

\maketitle

\textbf{Abstract} In this paper we investigate the properties of a semi-linear problem on a spin manifold involving the Dirac operator, through the construction of Rabinowitz-Floer homology groups. We give several existence results for sub-critical and critical non-linearities as application of the computation of the different homologies.
\tableofcontents
\section{Introduction}
Let $(M,g)$ be an $n$-dimensional Riemannian Spin manifold and $D$ the Dirac operator on its spinor bundle. It is tempting to try to solve semi-linear problems involving the Dirac operator. As a generic example, one can consider the problem $$Du=|u|^{p-1}u,$$ for $p\leq \frac{n+1}{n-1}$. 
This problem is variational and its energy functional $F$ can be written as $$F(u)=\frac{1}{2}\int_{M} <u,Du>-\frac{1}{p+1}\int_{M}|u|^{p+1}$$ also one can study the variations of $\frac{1}{2}\int_{M} <u,Du>$ under the constraint $||u||_{L^{p+1}}=1$. But the main problem here is that both those energy functionals are strongly indefinite since the operator $D$ has a sequence of eigenvalues unbounded from below and above. 
Hence a better method would be to develop a Morse theoretical argument that might give us informations about the critical points. Again the classical Morse theory is not applicable since the index and the co-index are infinite. One way of dealing with this problem is to do a Floer theory approach like in the work of Angenent and Van der Vost \cite{An} where they treat a system of super-quadratic equations.\\

In this paper we construct a sequence of homology groups related to the problem involving the Dirac operator stated above. Our approach is similar to the one used in \cite{Urs1} and \cite{Urs2}, where they construct a Rabinowitz-Floer homology relative to a modified energy functional that is stable under perturbation and this allows us to bring the computations to simple cases (the linear case in our work) to get an existence result for the general problem. 
Basically this process allows us to transform the semi-linear problem to a spectral theory problem with analysis of the spectrum of the operator $D$.\\

The Rabinowitz Floer homology was extensively used in \cite{Urs1} and \cite{Urs2} for the problem of periodic orbits of the Reeb vector field in the case of exact contact embedding. The original idea comes from the paper of P. Rabinowitz \cite{Rab} where he used a minimax approach on a functional under a constraint on the Hamiltonian to find periodic orbits of the Reeb vector field on convex hyper-surfaces of $\mathbb{R}^{n}$. The idea of the Floer-Rabinowitz homology is to include the Lagrange multiplier of the constraint in the arguments of the functional. This makes the perturbation of the Hamiltonian easier to understand and to control. A good example for its use would be the works of U. Frauenfelder \cite{Urs1}, \cite{Urs3} and F. Abbondandolo \cite{AB3}.

In our construction we consider the functional $E$ defined on a subspace to be mentioned later by $$E(u,\lambda)=\frac{1}{2}\int_{M}<Du,u>-\lambda\int_{M}(\frac{1}{p+1}|u|^{p+1}-1).$$ 
Notice that in the classical theory, if we take for instance the case of the Laplacian with Dirichlet boundary condition and we consider the problem $$E(u)=\frac{1}{2}\int_{\Omega}|\nabla u|^{2}$$ under the constraint $||u||_{L^{p+1}}=1$, then for $p$ sub-critical, one gets infinitely many critical points, now if we want to perturb the $p$, then the space of variation will vary as well, hence studying the stability of the solution will be a bit complicated. Though, in this case, it is not that complicated to perturb since the topology of the space of variations is the same under perturbation of $p$ and one can develop a Morse theoretical approach to link the critical points and the topology of the underlying space of variations. But in general, if the functional is strongly indefinite, it is not clear that a Morse homology (if it can be constructed) can be linked to the topology of the space of variations.\\

This problem was investigated by Isobe \cite{Iso1}, \cite{Iso2} under different assumptions, where he uses the idea of Rabinowitz \cite{Rab} again. 
We will see that using our method we recover most of those results and we get them in an easier way after we compute our homology and we get more extensions on the properties of the solutions.

\section{Preliminaries}
Let $(M,g)$ be an $n$-dimensional compact closed Riemannian manifold and we consider the $SO(n)$ principal bundle $P_{SO}(M)$ consisting of positively oriented frames on $(M,g)$. A spin structure on $(M,g)$ is a pair $\sigma=(P_{spin}(M),\mathcal{V})$, where $P_{Spin}(M)$ is a $Spin(n)$-principal bundle over  $M$ and $\mathcal{V}:P_{Spin}(M)\longrightarrow P_{SO}(M)$ such that the following diagram commutes :
\begin{center}
\begin{tikzpicture}
\matrix(m)[matrix of math nodes,
row sep=2.6em, column sep=2.8em,
text height=1.5ex, text depth=0.25ex]
{P_{spin}(M)\times Spin(n)&P_{spin}&M\\
P_{SO}(M)\times SO(n)&P_{SO}(M)&\\};
\path[->,font=\scriptsize,>=angle 90]
(m-1-1) edge node[auto] {} (m-1-2)
 edge node[auto] {$\mathcal{V}\times \Phi $} (m-2-1)
(m-1-2) edge node[auto] {$\mathcal{V}$} (m-2-2)
(m-2-1) edge node[auto] {} (m-2-2)
(m-2-2) edge node[auto] {} (m-1-3)
(m-1-2) edge node[auto] {} (m-1-3);
\end{tikzpicture}
\end{center}

Where $\Phi:Spin(n)\longrightarrow SO(n)$ is the non-trivial double covering of $SO(n)$. We consider now an $n$-dimensional complex representation $T:Spin(n)\longrightarrow End(\Sigma_{n})$ of $Spin(n)$. Here the vector space $\Sigma_{n}$ is of dimension $2^{\lfloor \frac{n}{2} \rfloor}$. The spinor bundle is then defined as the associate vector bundle $\Sigma (M,g,\sigma)=P_{Spin}(M)\times_{T} \Sigma_{n}$.
This bundle carries a natural Clifford multiplication, a hermitian metric and a natural metric connection. So we define the Dirac operator $D$ on this spinor bundle as the composition of the following : 
\begin{center}  
\begin{tikzpicture}
\matrix(m)[matrix of math nodes,
row sep=2.6em, column sep=2.8em,
text height=1.5ex, text depth=0.25ex]
{D:\Gamma(\Sigma M)&\Gamma(T^{*}M\otimes \Sigma M )&\Gamma(TM\otimes \Sigma M )&\Gamma(\Sigma M)\\};
\path[->,font=\scriptsize,>=angle 90]
(m-1-1) edge node[auto] {$\nabla^{\Sigma}$} (m-1-2)
(m-1-2) edge node[auto] {} (m-1-3)
(m-1-3) edge node[auto] {$Cliff$} (m-1-4);
\end{tikzpicture}
\end{center}
Now we recall a very important identity satisfied by this operator : 
\begin{proposition}[Schr\"{o}dinger-Lichnerowicz formula]
Let $(M,g,\sigma)$ be a compact Riemannian spin manifold and $R$ its scalar curvature, then we have $$D^{2}=\Delta + \frac{R}{4}.$$
\end{proposition}
For further information on Spin manifolds one can consult \cite{Fred} and \cite{Gi}. Now we give some analytical properties of the Dirac operator that we will use in what follow.

Let $D$ be the Dirac Operator of a compact spin manifold. Then we know that $D$ is essentially self adjoint as an operator in $L^{2}(\Sigma M)$ and admits a basis of smooth eigenspinors, that is $$L^{2}(\Sigma M)=\overline{\bigoplus_{\lambda \in \mathbb{R}} \ker(D-\lambda Id)}.$$
Moreover it has the space $H^{\frac{1}{2}}(\Sigma M)$ as its domain. Let $\{ \psi_{i}\}_{i\in \mathbb{Z}}$ be an $L^{2}$-orthonormal basis of eigenspinors with associated eigenvalues $(\lambda_{i})_{i\in \mathbb{Z}}$, then if $u\in L^{2}(\Sigma M)$ it has a representation in this basis as follows :$$u=\sum_{i\in \mathbb{Z}}a_{i}\psi_{i}.$$
Define the unbounded operator $|D|^{s}$ by $$|D|^{s}(u)=\sum_{i\in \mathbb{Z}} a_{i}|\lambda_{i}|^{2s}\psi_{i}$$
Using this operator we can define the inner product $$<u,v>_{s}=<|D|^{s}u,|D|^{s}v>_{L^{2}}.$$ This inner product induces a norm equivalent to the one in $H^{s}(\Sigma M)$. And so we will take $||u||^{2}=<u,u>_{\frac{1}{2}}$ as our standard norm.\\
In this work we will use the notation $\mathcal{H}$ for the space  $H^{\frac{1}{2}}(\Sigma M)$. For a given function $H(x,s_{x})$ where the $s$ is a spinor, we will write $h(x,s)$ for $\frac{\partial}{\partial s}H(x,s)$.
 
Consider the  functional $E : \mathcal{H} \times \mathbb{R} \rightarrow \mathbb{R}$ defined by : 
$$E(u,\lambda)=\frac{1}{2}\int_{M}<Du,u> -\lambda \int_{M}(H(x,u)-1).$$

We will be developing a Morse-Floer complex for this functional that will allow us to get informations on its critical points. The problem here is that the functional is strongly indefinite. Hence a classical Morse theory approach cannot be done. And as the reader can see, the way the functional was defined is similar to Rabinowitz approach for the Hamiltonian systems \cite{Rab}. This was also developed in \cite{Urs1}, \cite{Urs2} for exact contact embedding using the so called Rabinowitz-Floer homology.\\
In our case, we will restrict the perturbations to a class of non-linearities $H$ but we believe that this can be extended further more.\\
In the space of variations $\mathcal{H}\times \mathbb{R}$ we consider the norm defined by :$$||(u,\lambda)||^{2}=||u||^{2}+|\lambda|^{2}.$$
Consider the following assumptions:\\
 $(H1)$ There exist $1<p<\frac{n+1}{n-1}$, a constant $C$  and $r_{0}>0$ so that for $|s|>r_{0}>0$.  $$|\frac{\partial h(x,s)}{\partial s}|<C\left(1+|s|^{p-1}\right),$$\\
 $(H2)$ There exist two positive constants $0<c_{1}, c_{2}$, such that $$c_{1}|s|^{p+1}-c_{2}<<h(x,s),s>-H(x,s)$$  
 A typical $H$ that the reader can think of is $H(x,s)=f(x)|s|^{p+1}$ where $p$ is sub-critical and $f$ is a positive smooth function. In fact, assumption $(H1)$ is needed for the functional to be well defined and its Hessian to exist. In the other hand, assumption $(H2)$ is needed so that the functional $E$ satisfies the Palais-Smale compactness condition as we will show later on.\\ 
Now one can see that the critical points of $E$ satisfies the following system :
$$\left\{\begin{array}{ll}
Du=\lambda h(x,u)\\
\int_{M}H(x,u)=1
\end{array}
\right.$$
We will assume up to a small perturbation that $D$ has simple spectrum and that $E$ is Morse. (This is a generic assumption that we will remove in section 6). Notice that for $n=2,3,4 mod (8)$, this assumption means that $D$ cannot be the Dirac operator of any metric on $M$.\\
In all what follows, we assume that $M$ has positive scalar curvature, so that from the Schr\"{o}dinger-Lichnerowicz formula, all the eigenvalues of the Dirac operator $D$ are non zero (there are no harmonic spinors). This assumption is not really necessary, we introduce it just to make the exposition simpler for the reader, and in fact one can see how to do the modifications in the other case.
Our Main result can be stated as follow :
\begin{theorem}
Assume that $H$ satisfies $(H1)$ and $(H2)$ then the Rabinowitz-Floer homology is well defined.\\
Moreover, we have
$$H_{*}(H)=0.$$
\end{theorem}
\begin{theorem}
Under the assumptions of Theorem 2.2, If $H$ is $S^{1}$-invariant (resp. even), then the $S^{1}$-equivariant (resp. $\mathbb{Z}_{2}$-equivariant) homology is well defined and 
$$H_{*}^{S^{1}}(H)=\left\{\begin{array}{ll}
\mathbb{Z}_{2} \text{ if $*$ is even} \\
0 \text{ otherwise}
\end{array} \right.$$
Resp.
$$H_{*}^{\mathbb{Z}_{2}}(H)=\mathbb{Z}_{2}.$$
\end{theorem}

We will see in section 8 that from those two theorems follows many existence and multiplicity results on problems of the form $Du=h(x,u)$. We also have Brezis-Nirenberg type results in the critical case. We can apply the same procedure to the conformal Laplacian to get similar results in that case.

\section{Relative index and Moduli space of trajectories}
Notice that in our problem if we consider the Hessian of $E$ at a critical point $(\lambda,u)$, we get $$ Hess(E)(u,\lambda)=\left(\begin{array}{cc}
|D|^{-1}Du-\lambda |D|^{-1}\frac{\partial h}{\partial s}(x,u)) & -|D|^{-1}h(x,u)\\
-|D|^{-1}h(x,u) & 0
\end{array}
\right)
$$
Hence the index and co-index of the critical point are infinite. So we need to introduce an alternative way of grading as in \cite{AB2}. 
First we consider the splitting of $\mathcal{H}$ as $$\mathcal{H}=\mathcal{H}^{+}\oplus\mathcal{H}^{-} $$ where $\mathcal{H}^{-}=\overline{span\{\varphi_{i},i<0 \}}$ and $\mathcal{H}^{+}=\overline{span\{\varphi_{i},i>0 \}}$. We write for every $u\in \mathcal{H},$ $u=u^{+}+u^{-}$ according to the previous splitting.
\begin{definition}
Consider two closed subspaces $V$ and $W$ of a Hilbert space $F$. We say that $V$ is a compact perturbation of $W$ if $P_{V}-P_{W}$ is a compact operator. 
\end{definition}
$P_{V}$ in the previous definition denote the orthogonal projection on $V$. If in the case where $V$ is a compact perturbation of $W$, we can define the relative dimension as $$dim(V,W)=dim(V\cap W^{\perp})-dim(V^{\perp}\cap W).$$
One can check that it is well defined and finite. Now if we have three subspaces $V$, $W$ and $H$ such that $V$ and $W$ are compact perturbations of $H$. Then $V$ is also a compact perturbation of $W$ and $$dim(V,W)=dim(V,H)+dim(H,W).$$
Using this concept of relative dimension we can define a relative index as our grading.  
\begin{lemma}
The relative index is well defined for critical points of $E$ as long as $H$ satisfies $(H1)$, under assumption that it is Morse. 
\end{lemma}
\begin{proof}
Let $\Gamma=\mathcal{H}^{-}\times \mathbb{R}$, and $(u,\lambda)$ a critical point of $E$. Notice now that the operator $f\mapsto Df-\lambda \frac{\partial h}{\partial s}(x,u)f$ has discrete spectrum since $D^{-1}$ is a compact operator. So the space $V^{-}(u,\lambda)$ which is the closure of the span of the eigenfunction of the Hessian of $E$ at the point $(u,\lambda)$, corresponding to negative eigenvalues, is well defined. Moreover, $V^{-}(u,\lambda)$ is a compact perturbation of $\Gamma$. This follows from the fact that $$\left( \begin{array}{cc}
|D|^{-1}D & |D|^{-1} \\
|D|^{-1} & 0
\end{array}
\right) - \left( \begin{array}{cc}
|D|^{-1}D-\lambda |D|^{-1}\frac{\partial h}{\partial s}(x,u) & |D|^{-1}h(x,u) \\
|D|^{-1}h(x,u) & 0
\end{array}
\right),$$
is a compact operator.
\end{proof}
We define now the moduli space of $\mathcal{H}$-gradient trajectories. For that consider the following differential system :
\begin{equation}\label{GF}
\left\{\begin{array}{ll}
\frac{\partial u}{\partial t}=u^{-}-u^{+}+\lambda |D|^{-1}h(x,u)\\
\frac{\partial \lambda}{\partial t}=\int_{M}H(x,u)-1
\end{array}
\right.
\end{equation}
This system is in fact the descending gradient flow of our  functional in $\mathcal{H}\times \mathbb{R}$ and since the right hand side is smooth, then we have local existence of the flow. Notice that one is tempted to use the $L^{2}$ gradient flow, to get a heat flow equation, but in this case the problem is ill-posed since the spectrum of $D$ is unbounded from below. This type of gradient flow was used by A. Bahri in \cite{B}, where it appears that it has better properties than the classical heat flow and it is moreover positivity preserving. Given two critical points $z_{0}=(u_{0},\lambda_{0})$ and $z_{1}=(u_{1},\lambda_{1})$ such that $E(z_{i})\in [a,b]$ for $i=0,1$. We define the space of connecting orbits from $z_{0}$ to $z_{1}$ by  $$\mathbb{M}^{a,b}(z_{0},z_{1})=\left\{z\in C^{1}(\mathbb{R},\mathcal{H}\times \mathbb{R}) |z \text{ satisfies \ref{GF} and } z(-\infty)=z_{0};z(+\infty)=z_{1}\right\}.$$
The moduli space of trajectories is then defined by $$\mathcal{M}^{a,b}(z_{0},z_{1})=\mathbb{M}^{a,b}(z_{0},z_{1})/\mathbb{R}.$$
\begin{proposition}
Assume that $i_{rel}(z_{0})> i_{rel}(z_{1})$, then if $E$ is Morse-Smale, $\mathcal{M}^{a,b}(z_{0},z_{1})$ is a finite dimensional manifold of dimension $i_{rel}(z_{0})-i_{rel}(z_{1})-1$.
\end{proposition}
\begin{proof}
One can see first that $\mathbb{M}^{a,b}(z_{0},z_{1})=F^{-1}(0)$ where $F:C^{1}(\mathbb{R},\mathcal{H}\times \mathbb{R})\mapsto \mathcal{Q}^{0}=C^{0}(\mathbb{R},\mathcal{H}\times \mathbb{R})$ defined by $$F(z)=\frac{dz}{dt}+\nabla E(z),$$
We want to use the implicit function theorem to prove our result. For that we need to show that the linearised operator $F$ is Fredholm and onto.
The linearised operator corresponds to $\partial F(z)=\frac{d}{dt}+Hess(E(z))$, and this is a linear differential equation in the Banach space $\mathcal{H}\times \mathbb{R}$. Now we refer to \cite{AB1} for the details of the proof. To see that we first need to notice that $Hess(E(z))$ can be written as $$Hess(E(z))=\left( \begin{array}{cc}
|D|^{-1}D & 0\\
0 & 1
\end{array}
\right) +\left( \begin{array}{cc}
-\lambda |D|^{-1}\frac{\partial h}{\partial s}(x,u) & -|D|^{-1}h(x,u)\\
-|D|^{-1}h(x,u) & -1
\end{array}
\right). 
$$
The operator $\left( \begin{array}{cc}
|D|^{-1}D & 0\\
0 & 1
\end{array}
\right)$ is time independent and hyperbolic and the operator $\left( \begin{array}{cc}
-\lambda |D|^{-1}\frac{\partial h}{\partial s}(x,u) & -|D|^{-1}h(x,u)\\
-|D|^{-1}h(x,u) & -1
\end{array}
\right) $ is compact. Hence we have that $DF$ is a Fredholm operator with index $$ind(DF(z))=dim(V^{-}(F(z_{0}),V^{-}(F(z_{1}))$$
$$=dim(V^{-}(F(z_{0}),\Gamma)+dim(\Gamma,V^{-}(F(z_{1}))$$
$$=i_{rel}(z_{0})-i_{rel}(z_{1}).$$
Also from the same work \cite{AB1}, we have the fact that $DF(z)$ is onto if and only if the intersection is transverse.\\
To finish the proof now, it is enough to notice that the action of $\mathbb{R}$ is free and hence we can mod by that action to get the desired result.
\end{proof}
\section{Compactness}
\subsection{(PS) Condition}
We recall that a functional $F$ is said to satisfies the (PS) condition ((PS) for Palais-Smale), at the level $c$ if every sequence $(x_{k})$ such that $F(x_{k})\longrightarrow c$ and $Df(x_{k})\longrightarrow 0$, has a convergent subsequence. We will say that $F$ satisfies (PS) if the previous condition is satisfied for all $c\in \mathbb{R}$.
\begin{lemma}
If $H$ satisfies $(H2)$ then the functional $E$ satisfies (PS).
\end{lemma}
\begin{proof}
Let $z_{k}=(u_{k},\lambda_{k})$ be a (PS) sequence for $E$, then the following holds
\begin{equation} 
Du_{k}-\lambda_{k}h(u_{k})=o(1), \label{grad}
\end{equation}
\begin{equation}
\int_{M}(H(u_{k})-1)=o(1),
\end{equation}
and
\begin{equation}\label{en}
\frac{1}{2}\int u_{k}Du_{k}-\lambda_{k}\int_{M}(H(u_{k})-1=c+o(1) 
\end{equation}
Composing (\ref{grad}) with $u_{k}$ and using (\ref{en}) we get $$\lambda_{k}(\int_{M}(<h(u_{k}),u_{k}>-2H(u_{k}))+2)=c+o(||u_{k}||)$$
Therefore $$\lambda_{k}\int_{M}<h(u_{k}),u_{k}>=C+o(|\lambda_{k}|)+o(||u_{k}||)$$
Therefore, from $(H2)$ we have that $$\int_{M}<h(x,u),u>\geq \int_{M}H(x,u)-c_{2}+c_{1}\int_{M}|u|^{p+1},$$
So either we have the boundedness of $u_{k}$ in $L^{p+1}$ or we can assume without loss of generality that $\int_{M}<h(x,u_{k}),u_{k}> >0$. Thus,
\begin{equation} \label{lam b}
\lambda_{k}=C+o(||u_{k}||)+o(|\lambda_{k}|).
\end{equation}

Now we compose (\ref{grad}) with $u_{k}^{+}$ to get 
\begin{equation}\label{u+}
||u_{k}^{+}||^{2}=\lambda_{k} \int_{M}<h(u_{k}),u_{k}^{+}>+o(||u_{k}^{+}||)
\end{equation}
In a similar way we get the inequality
\begin{equation}\label{u-}
||u_{k}^{-}||^{2}=\lambda_{k} \int_{M}<h(u_{k}),u_{k}^{-}>+o(||u_{k}^{-}||)
\end{equation}
 Now combining (\ref{u+}) and (\ref{u-}), we have $$||u_{k}||^{2}=\lambda_{k} \int_{M}<h(u_{k}),u_{k}>+o(||u_{k}||)$$
Now using (\ref{lam b}) we get $$||u_{k}||^{2}=C+o(|\lambda_{k} |)+o(||u_{k}||)$$

Hence $||u_{k}||$ is bounded in $H^{\frac{1}{2}}(\Sigma M)$ and $\lambda_{k}$ is bounded in $\mathbb{R}$. By compactness of the Sobolev Embedding for $p<\frac{n+1}{n-1}$ we have the strong convergence up to a subsequence of  $(u_{k})$ to a function $u$ in $L^{p+1}$ and weakly in $H^{\frac{1}{2}}$ also the convergence of $\lambda_{k}$ to $\lambda$. So going back to (\ref{grad}) and again multiplying by $u_{k}^{+}$ and $u_{k}^{-}$ we get the convergence in norm and hence the (PS) condition holds. 
\end{proof}
\subsection{Compactification of the Moduli spaces}
In order to define the homology we need that the boundary operator (that will be defined later) satisfies $\partial ^{2}=0$. The main step in showing this fact is to prove the compactness of the moduli space of trajectories between two critical points with index difference equal to $2$.
So let us consider $z(t)\in \mathbb{M}^{a,b}(z_{0},z_{1})$. By definition $z=(u,\lambda)$ satisfies \begin{equation} \label{eq1}
u'=u^{-}-u^{+}+\lambda |D|^{-1}h(x,u)
\end{equation}
and \begin{equation} \label{eq2}
\lambda ' =-\int_{M}H(x,u)-1
\end{equation}
First since $z=(u,\lambda)$ is a gradient flow, then $$\int_{-\infty}^{+\infty}||z'(t)||^{2}dt\leq b-a,$$
Now considering $$<z(t),z'(t)>+2E(z)=\lambda \int_{M} (<h(x,u),u>-H(x,u)+1)$$

hence $$\int_{T}^{T+1} |\lambda| |\int_{M} (<h(x,u),u>-H(x,u)+1)| \leq C \sup_{t}||z(t)||$$
using the assumption $(H2)$, we have 
\begin{equation}\label{est}
\int_{T}^{T+1}|\lambda |\int_{M}|u|^{p+1}dt\leq C(1+\sup_{t}||z(t)||)
\end{equation}
In the other hand, If we let $P_{+}$ and $P_{-}$ to be the projectors on $\mathcal{H}^{+}$ and $\mathcal{H}^{-}$ respectively, then the operator $$G(t)=e^{-t}P_{-}\chi_{t\geq 0}-e^{t}P_{+}\chi_{t\leq 0}$$ is a fundamental solution for the evolution operator $\frac{d}{dt}+P_{+}-P_{-}$. Hence one has the following estimate $$||G(t)||\leq Ce^{-|t|}.$$
In particular this leads to $$u(t)=\int_{-\infty}^{+\infty}G(t-s)\lambda(s)|D|^{-1}h(x,u)ds$$
from this we get that $$||u(t)||\leq C(1+\sup_{T}\int_{T}^{T+1}\lambda |||D|^{-1}h(x,u)||)$$
But using the Sobolev embedding and assumption $(H2)$ we get $$|||D|^{-1}h(x,u)||\leq |||D|^{\frac{-1}{2}}h(x,u)||_{L^{2}} \leq C(1+||u||_{L^{p+1}}^{p})$$
Now from (\ref{est}) we have $$||u(t)|| +|\lambda| \leq C(1+\sup_{t}||z(t)||^{\frac{p}{p+1}})$$
This gives us a uniform bound on $||z(t)||$.
This is enough to get convergence in the $C^{1}_{loc}(\mathbb{R},\mathcal{H})$ but in fact one can get more.
\begin{lemma}
Let $z$ be a solution of (\ref{eq1})-(\ref{eq2}) that is bounded uniformly in $\mathcal{H}\times \mathbb{R}$, then $u$ is bounded in $C^{0,\alpha}$ for some $0<\alpha<1$.
\end{lemma}
\begin{proof}
The idea here is to show that the solution operator has a regularisation effect and by a bootstrap argument we get the desired result.
So let us recall that $$u(t)=\int_{-\infty}^{+\infty}G(t-s)\lambda(s)|D|^{-1}h(x,u(s))ds,$$
now if $u\in \mathcal{H}$, then $h(x,u(s))\in L^{\frac{2n}{p(n-1)}}$ hence $|D|^{-1}h(x,u(s))\in W^{1,\frac{2n}{p(n-1)}}$. It is important also to notice that $P_{+}$ and $P_{-}$ preserves the regularity. 
Hence if $z$ is a solution then $||u||_{W^{\frac{2n}{p(n-1)}}}\leq C ||u||$.
\end{proof}

Notice that the moduli spaces are modelled on the affine space $\mathcal{Q}^{1}(z_{0},z_{1})=\tilde{z}+C^{1}(\mathbb{R},X)$, where $X=C^{0,\alpha}\times \mathbb{R}$. We consider the map $ev:\mathcal{M}(z_{0},z_{1})\longmapsto X$ defined by $ev(z)=z(0)$. This map is onto and hence the set $\mathcal{M}(z_{0},z_{1})$ is precompact.\\
\subsubsection{Compactification by broken trajectories}
Here we recall the operator $T_{i,i+1}:\mathcal{Q}^{1}(z_{i},z_{i+1})\longrightarrow \mathcal{Q}^{0}$ defined by $$T_{i}(z)=\frac{dz}{dt}+\nabla E(z)$$  is Fredholm assuming transversality and $\mathcal{M}(z_{i},z_{i+1})=T_{i,i+1}^{-1}(0)$. These operators then are surjective and hence they admit a right inverse $S_{i,i+1}$. 
Let $z_{01}\in\mathcal{M}(z_{0},z_{1})$ and $z_{12}\in\mathcal{M}(z_{1},z_{2})$. We define the function $$z_{02,T}(t)=(1-\varphi(\frac{t}{T}))z_{01}(t+2T)+\varphi(\frac{t}{T}z_{12}(t-2T)$$
for $\varphi$ a non-negative function such that $\varphi(t)=0$ if $t<-1$ and $\varphi(t)=1$ for $t\geq 1$.\\

We define now the operator $$A_{T}=R^{+}_{T}\tau_{2T}S_{01}\tau_{-2T}R^{+}_{T}+R^{-}_{T}\tau_{2T}S_{12}\tau_{-2T}R^{-}_{T}$$
where $\tau$ is the translation operator defined by $\tau_{a}f(t)=f(t+a)$ and $R^{\pm}$ is a pair of smooth functions satisfying $(R^{+}_{1})^{2}+(R^{-}_{1})^{2}=1$ and $R^{+}_{1}(t)=0$ for $t\leq -1$, $R^{+}_{1}(t)=R^{-}_{1}(-t)$ and $R^{\pm}_{T}=R^{\pm}(\frac{t}{T})$. 
Then we have that $dT_{02}(z_{02,T})\circ A_{T}$ converges to the identity operator as $T\longrightarrow \infty$. Hence by setting $z=z_{02,T}+A_{T}w$, finding a connecting orbit is equivalent to solving $T_{02}(z)=0$ which can be done using the implicit function theorem in the space $\mathcal{Q}^{0}$. This is of course possible because the operator $T_{02}$ is Fredholm. Notice that this construction can be done transversally to the kernel of the linearized operator by setting $z=z_{02,T}+u+A_{T}w$ where $u$ is an element in the kernel and $w$ is small.\\
Using this we have defined the gluing map by $$z_{01}\sharp_{T,v}z_{12}=z_{02,T}+u+A_{T}w.$$
Now we can deduce that in fact the set $\mathbb{M}^{a,b}(z_{0},z_{1})$ has compact closure.
\subsection{Construction of the Homology}

In this section we will define the different chain complexes and their homologies and we will give an explicit computation later, of the later mentioned homologies under specific assumptions.\\
Given a potential $H$ satisfying $(H1)$ and $(H2)$, we let $E_{H}$ denote its energy functional and $E_{H_{0}}$ for $H_{0}(s)=\frac{1}{2}|s|^{2}$. For $a<b$ we define the critical sets $Crit_{k}^[a,b](E_{H})$ as the set of critical points of $E_{H}$ with energy in the interval $[a,b]$ and relative index $k$.\\
Notice that if $E_{H}$ is Morse and satisfies (PS) (which we can always assume as we will see in section 6), then $Crit_{k}^{[a,b]}(E_{H})$ is always finite. Now we define the chain complex $C_{k}^{[a,b]}(H)$ as the vector space over $\mathbb{Z}_{2}$ generated by $Crit_{k}^{[a,b]}$, for every $k\in \mathbb{Z}$. That is $$C_{k}^{[a,b]}(H)=Crit_{k}^{[a,b]}(E_{H})\otimes \mathbb{Z}_{2},$$ 
And the boundary operator $\partial$ defined for $z \in Crit_{k}^{[a,b]}(E_{H})$ by 
$$\partial z =\sum_{y\in Crit_{k-1}^{[a,b]}(E_{H})} (\sharp  \mathcal{M}(z,y) mod [2])y$$
Using the compactness result of the previous subsection we do have $\partial^{2}=0$ and therefore it is indeed a chain complex and we will write $H_{*}^{[a,b]}(H)=H_{*}(C_{*}^{[a,b]}(H),\partial)$ and just $H_{*}^{[a,b]}$ for the case $H_{0}(s)=|s|^2$.\\
\subsubsection{The Equivariant case}
Here we will define the equivariant homology for two different group actions, the first one that is needed in our problem is the $S^{1}$ action and the second one is the $\mathbb{Z}_{2}$ action, this last one will be useful when we will investigate the classical Yamabe problem.\\
In the case where $H$ is $S^{1}$-invariant, that is $H(x,e^{i\theta}s)=H(x,s)$, we define the chain complex $$C_{k}^{[a,b],S^{1}}(H)=\frac{Crit_{k}^{[a,b]}(E_{H})}{S^{1}}\otimes \mathbb{Z}_{2}.$$
Notice that this definition makes sense since the critical points here are in fact critical circles since $E_{H}$ is equivariant. Notice now that by breaking the symmetry, that each $z_{k}\in \frac{Crit_{k}^{[a,b]}(E_{H})}{S^{1}}$ splits to a max and a min $z_{k}^{+}\in Crit_{k+1}^{[a,b]}(\tilde{E}_{H})$ and $z_{k}^{+}\in Crit_{k}^{[a,b]}(\tilde{E}_{H})$ where $\tilde{E}_{H}$ is the perturbed functional. We define for $z_{k+1}\in \frac{Crit_{k}^{[a,b]}(E_{H})}{S^{1}}$, $$\partial_{S^{1}} z_{k+1}=\sum_{z_{k}\in \frac{Crit_{k}^{[a,b]}(E_{H})}{S^{1}}} (<z_{k+1}^{+},z_{k}^{+}>z_{k}$$
where here $<x,y>=\sharp ( \mathcal{M}(x,y))mod[2]$. We see here that $\partial_{S^{1}}$ is well defined and to show that indeed it is a boundary operator we need to prove that
\begin{lemma}
$\partial_{S^{1}}^{2}=0$
\end{lemma}
\begin{proof}
First we define the following chain complex $$\overline{C}_{k}=\bigoplus_{z_{k}\in \frac{Crit_{k}^{[a,b]}(E_{H})}{S^{1}}} (z_{k}^{+},z_{k}^{-})\otimes \mathbb{Z}_{2},$$
with the boundary operator $$\overline{\partial} (z_{k+1}^{+},z_{k+1}^{-}) =\sum_{z_{k}\in \frac{Crit_{k}^{[a,b]}(E_{H})}{S^{1}}} (<z_{k+1}^{+},z_{k}^{+}>z_{k}^{+},<z_{k+1}^{-},z_{k}^{-}>z_{k}^{-}).$$
We claim that $\overline{\partial}^{2}=0$. Indeed this follows from the computation of $\partial \mathcal{M}(z_{k+1}^{+},z_{k-1}^{+})$. This later boundary contains two kind of terms, the ones of the form $\mathcal{M}(z_{k+1}^{+},\overline{z}_{k+1}^{-})$, $\mathcal{M}(\overline{z}_{k+1}^{-},z_{k-1}^{+})$ and those of the form $\mathcal{M}(z_{k+1}^{+},z_{k}^{+})$, $\mathcal{M}(z_{k}^{+},z_{k-1}^{+})$. The second terms are the ones that appear in the formula for $\overline{\partial}^{2}$. So it is enough to show that the terms of the first kind cancel.\\
To show this we notice that first $\sharp  \mathcal{M}(z_{k+1}^{+},z_{k+1}^{-}) mod [2] =0$ and if $\sharp  \mathcal{M}(z_{k+1}^{+},\overline{z}_{k+1}^{-}) \not =0$, then by the $S^{1}$ action we have that $\sharp  \mathcal{M}(z_{k+1}^{+},\overline{z}_{k+1}^{+}) \not =0$ which is impossible by transversality, hence  $\sharp  \mathcal{M}(z_{k+1}^{+},\overline{z}_{k+1}^{-}) =0$ and this finishes the proof of the claim that $(\overline{C}_{*},\overline{\partial})$ is a chain complex.\\
In fact this cancellation caused by the $S^{1}$ action is exactly like the $\Delta$ operator in the loop space introduced by Denis and Sullivan in \cite{CS}. This operator satisfies $\Delta^{2}=0$ and same holds in our case.

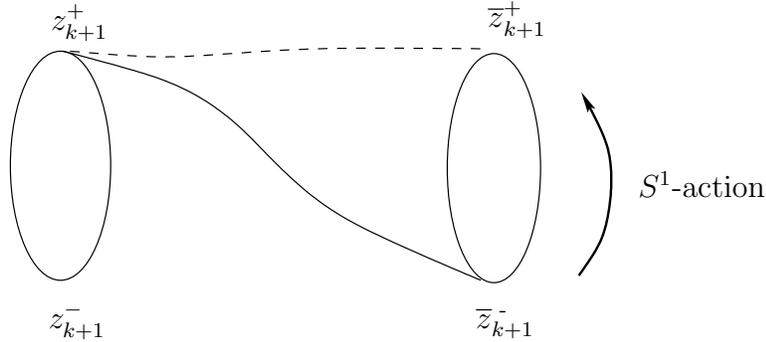
\begin{figure}[H]
\centering
\begin{picture}(0,0)%
\includegraphics{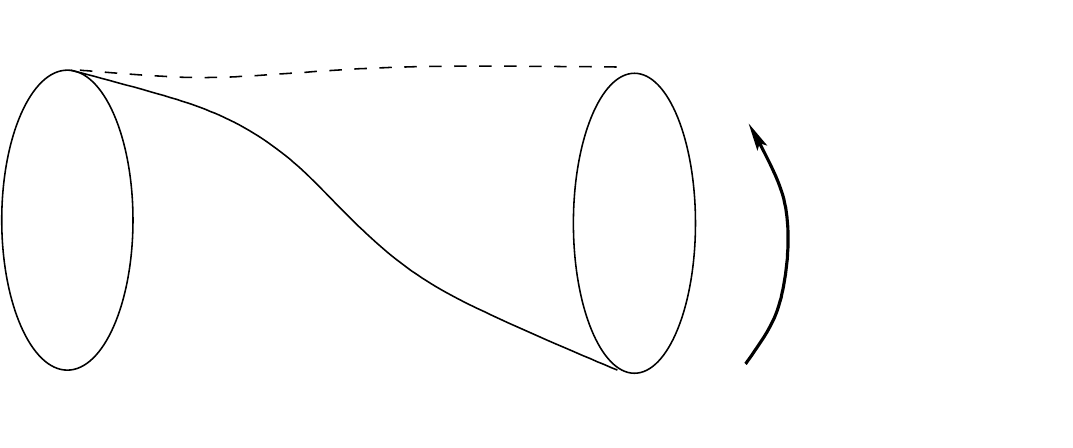}%
\end{picture}%
%
%
\setlength{\unitlength}{3947sp}%
\begingroup\makeatletter\ifx\SetFigFont\undefined%
\gdef\SetFigFont#1#2#3#4#5{%
  \reset@font\fontsize{#1}{#2pt}%
  \fontfamily{#3}\fontseries{#4}\fontshape{#5}%
  \selectfont}%
\fi\endgroup%
\begin{picture}(5137,2137)(1508,-2115)
\put(1756,-2041){\makebox(0,0)[lb]{\smash{{\SetFigFont{12}{14.4}{\rmdefault}{\mddefault}{\updefault}{\color[rgb]{0,0,0}$z_{k+1}^{-}$}%
}}}}
\put(1771,-166){\makebox(0,0)[lb]{\smash{{\SetFigFont{12}{14.4}{\rmdefault}{\mddefault}{\updefault}{\color[rgb]{0,0,0}$z_{k+1}^{+}$}%
}}}}
\put(4441,-2041){\makebox(0,0)[lb]{\smash{{\SetFigFont{12}{14.4}{\rmdefault}{\mddefault}{\updefault}{\color[rgb]{0,0,0}$\overline{z}_{k+1}^{\text{ -}}$}%
}}}}
\put(4516,-151){\makebox(0,0)[lb]{\smash{{\SetFigFont{12}{14.4}{\rmdefault}{\mddefault}{\updefault}{\color[rgb]{0,0,0}$\overline{z}_{k+1}^{+}$}%
}}}}
\put(5461,-1216){\makebox(0,0)[lb]{\smash{{\SetFigFont{12}{14.4}{\rmdefault}{\mddefault}{\updefault}{\color[rgb]{0,0,0}$S^{1}$-action}%
}}}}
\end{picture}%
\caption{Cancelation from the $S^{1}$ action}
\end{figure}

We consider now the map $f_{*}$ defined on the chain complex $(\overline{C}_{*},\overline{\partial})$ by  $f_{*}((z_{k}^{+},z_{k}^{-})=z_{k}$. $f$ is well defined and it is an isomorphism between $C_{*}^{[a,b],S^{1}}(H)$ and $(\overline{C}_{*},\overline{\partial})$. And by the $S^{1}$ action we have that $<z_{k+1}^{+},z_{k}^{+}>=<z_{k+1}^{-},z_{k}^{-}>$ we have that $\partial_{S^{1}}=f^{-1}\circ \overline{\partial} \circ f$. Which completes the proof of the lemma.
\end{proof}
So now we define the equivariant homology by $H_{*}^{[a,b],S^{1}}(H)=H_{*}(C_{*}^{[a,b],S^{1}}(H),\partial_{S^{1}})$.\\
The $\mathbb{Z}_{2}$-equivariant homology is easier to define, indeed we assume that $H$ is even, then the functional $E_{H}$ is invariant under the the obvious $\mathbb{Z}_{2}$-action. The critical points in this case come in pairs $z_{k}=(\lambda_{k},u_{k})$ and $\overline{z}_{k}=(\lambda_{k},-u_{k})$. We consider then the chain complex $$C_{k}^{[a,b],\mathbb{Z}_{2}}(H)=\frac{Crit_{k}^{[a,b]}(E_{H})}{\mathbb{Z}_{2}}\otimes \mathbb{Z}_{2},$$
along with the boundary operator $\partial_{\mathbb{Z}_{2}}$ defined for $z_{k+1}\in \frac{Crit_{k+1}^{[a,b]}(E_{H})}{\mathbb{Z}_{2}}$ by $$\partial_{\mathbb{Z}_{2}}z_{k+1}=\sum_{z_{k}\in \frac{Crit_{k+1}^{[a,b]}(E_{H})}{\mathbb{Z}_{2}}} (<z_{k+1},z_{k}>+<z_{k+1},\overline{z}_{k}>)z_{k}.$$
Notice now that the quotient map $f:Crit_{*}^{[a,b]}(E_{H})\longrightarrow\frac{Crit_{*}^{[a,b]}(E_{H})}{\mathbb{Z}_{2}}$ extends to a map from $C_{*}^{[a,b]}(H)$ onto $C_{*}^{[a,b],\mathbb{Z}_{2}}(H)$. Moreover, it is easy to show that the following diagram commutes:\\
\begin{center}

\begin{tikzpicture}
\matrix(m)[matrix of math nodes,
row sep=2.6em, column sep=2.8em,
text height=1.5ex, text depth=0.25ex]
{C_{k+1}^{[a,b]}(H)&C_{k}^{[a,b]}(H)\\
C_{k+1}^{[a,b],\mathbb{Z}_{2}}(H)&C_{k}^{[a,b],\mathbb{Z}_{2}}(H)\\};
\path[->,font=\scriptsize,>=angle 90]
(m-1-1) edge node[auto] {$\partial$} (m-1-2)
edge node[auto] {$f_{k+1}$} (m-2-1)
(m-1-2) edge node[auto] {$f_{k}$} (m-2-2)
(m-2-1) edge node[auto] {$\partial_{\mathbb{Z}_{2}}$} (m-2-2);
\end{tikzpicture}
\end{center}
Hence we have indeed a chain complex and we will write the corresponding homology as $H_{*}^{[a,b],\mathbb{Z}_{2}}(H)=H_{*}(C_{*}^{[a,b],\mathbb{Z}_{2}}(H),\partial_{\mathbb{Z}_{2}})$.

\section{Stability}
In this section we will consider two "Hamiltonians" $H_{1}$ and $H_{2}$ and we will show that that under suitable conditions, $H^{*}(H_{1})=H_{*}(H_{2})$. The proof will be done in the general case and there is absolutely no difference in the equivariant case since all the perturbations can be taken equivariant. Let $\eta(s)$ be a smooth function on $\mathbb{R}$ such that $\eta(s)=1$ for $s\geq 1$ and $\eta(s)=0$ if $s\leq 0$. We set $H_{s}=(1-\eta(s))H_{1}+\eta(s)H_{2}$.\\
Now we define the non-autonomous gradient flow by $$z'(t)=-\nabla E_{H_{t}}(z(t)),$$
Where $\nabla E_{H_{t}}$ is the gradient with respect to $z$ for a fixed $t$. Given $z_{1}$ a critical point of $E_{H_{1}}$ and $z_{2}$ a critical point of $E_{H_{2}}$, we let $z(t)$ the flow line from $z_{1}$ to $z_{2}$.
\begin{lemma}
There exists $\delta>0$ such that if $|H_{1}-H_{2}|\leq \delta$ then $z(t)$ is uniformly bounded by a constant depending only on $z_{1}$ and $z_{2}$.
\end{lemma}
\begin{proof}
The proof here is not very different from the one done in subsection 4.2. Indeed here one needs to worry about the boundedness of $\lambda$ along the flow.\\
First notice that $$\frac{\partial E_{H_{t}}(z(t))}{\partial t}=-||z'(t)||^{2}+\lambda \eta'(t) \int_{M}H_{1}-H_{2}.$$
Therefore, we have $$E_{H_{t}}(z(t))\leq E_{H_{1}}(z_{1})+\delta \int_{0}^{t} \lambda \eta'(s)$$
and $$\int_{-\infty}^{+\infty}||z'(t)||^{2}dt\leq E_{H_{2}}(z_{2})-E_{H_{1}}(z_{1})+\delta\int_{0}^{1}\lambda(t)dt.$$
Using again the identity
$$<z(t),z'(t)>+2E_{H_{t}}(z)=\lambda \int_{M} (<h_{t}(x,u),u>-H_{t}(x,u)+1),$$
We find that 

$$\int_{T}^{T+1}|\lambda| \int_{M} (<h_{t}(x,u),u>-H_{t}(x,u)+1)dt\leq \sup_{t}||z(t)||(C_{1}+\delta ||\lambda||_{\infty})+\delta ||\lambda||_{\infty}+C_{2},$$
This can be written as 
$$\int_{T}^{T+1}|\lambda| \int_{M} (<h_{t}(x,u),u>-H_{t}(x,u)+1)dt\leq C\sup_{t}(||z(t)||+\delta ||z(t)||^{\frac{3}{2}}),$$
where $C$ is a constant depending on $z_{1}$ and $z_{2}$.\\
The rest of the argument from subsection 4.2 can be carried out to reach the following estimate :

$$||z(t)||\leq C(1+\sup_{t}(||z(t)||+\delta||z(t)||^{\frac{3}{2}})^{\frac{p}{p+1}}),$$
Thus, since $\frac{3p}{2(p+1)}\leq 1$, we have for $\delta<\delta_{0}(z_{1},z_{2})$, that $||z(t)||$ is uniformly bounded. 
\end{proof}
Notice that in the previous proof the need of $\delta$ to be small is only required in the case $n=2$.\\

Now, this uniform boundedness implies pre-compactness as in the autonomous case, hence we define now the moduli space of trajectories of the non-autonomous gradient flow, $\mathbb{M}(z_{1},z_{2})$, is pre-compact and we omit here the similar gluing construction that can be done to compactify. In fact we can show by following the same proof that $\mathbb{M}(z_{1},z_{2})$ is a finite dimensional manifold with dimension $i_{rel}(z_{1})-i_{rel}(z_{2})$.\\
And if $i_{rel}(z_{1})-i_{rel}(z_{2})=1$, we have that $$\partial \mathbb{M}(z_{1},z_{2})=\bigcup_{x\in Crit_{i_{rel}(z_{2})}(E_{H_{1}})} \mathbb{M}(z_{1},x)\times \mathbb{M}(x,z_{2})\bigcup_{y\in Crit_{i_{rel}(z_{1})}(E_{H_{2}})}\mathbb{M}(z_{1},y)\times \mathbb{M}(y,z_{2}).$$
With this in mind we can construct the continuation Isomorphism $$\Phi_{12} : C_{*}(H_{1})\longrightarrow C_{*}(H_{2}),$$ defined in the chain level by 
$$\Phi_{12}(z)=\sum_{x\in Crit_{i_{rel}(z)}(E_{H_{2}})}(\sharp \mathbb{M}(z,x) mod [2])x$$

Now by the previous remark on the boundary of the moduli space in the non-autonomous case, one sees that $\partial_{1} \Phi_{12}+\Phi_{12} \partial_{2}=0$, this shows that it is a chain homomorphism, hence it descends to the homology level.\\
The last thing to check is that it is an isomorphism. And that is by taking a homotopy of homotopies and doing the same work again (as in Schwarz \cite{Sh}).\\
Therefore we have that that $H_{*}(E_{H_{1}})=H_{*}(E_{H_{2}})$.

\begin{proposition}
Assume that $H_{1}$ and $H_{2}$ satisfies the assumptions $(H1)$ and $(H2)$. Then $H_{*}(H_{1})=H_{*}(H_{2})$. 
\end{proposition} 
\begin{proof}
Without loss of generality one can assume the existence of a homotopy $H_{s}$ for $s\in [0,1]$ and a partition $0<s_{1}<\cdots <s_{k}<1$, such that $|H_{s_{i+1}}-H_{s_{i}}|\leq \delta$ hence there exists a continuation isomorphism $\Phi_{i,i+1}$ between $H_{*}(H_{s_{i}})$ and $H_{*}(H_{s_{i+1}})$ and since there exist finitely many of them, one gets an isomorphism between $H_{*}(H_{1})$ and $H_{*}(H_{2})$.
\end{proof}
The same stability results hold for the equivariant case.
\section{Transversality}
In this section we will show that up to a small smooth perturbation of the Hamiltonian we can assume that $E_{H}$ is Morse. Then, It can be approximated by a Morse-Smale functional with the same critical points and the same connections.

\begin{lemma}
Consider a Hamiltonian $H$ satisfying $(H1)$ and $(H2)$ then for a generic perturbation $K$ of $H$ in $C^{3}(\Gamma(\Sigma M))$, the energy functional $E_{H+K}$ is Morse.
\end{lemma} 
\begin{proof}
We consider the functional $\psi :\mathcal{H}\times \mathbb{R}\times C^{3}(\Gamma(\Sigma M))  \longrightarrow \mathcal{H}\times \mathbb{R}$ defined by $$\psi(z,K)=\nabla E_{H+K}(z).$$
Notice first the inverse image of zero corresponds to critical points of the functional with Hamiltonian
$H+K$. Also, for $(z,K)\in \psi^{-1}(0)$, $$\partial_{z} \psi(z,h)v= Hess(E_{H+K}(z))v,$$
which is a perturbation of a compact operator, and hence it is a zero index Fredholm operator.
Now it remains to show that $\nabla \psi (z,K)$ is surjective. So let us compute the differential with respect to $K$ : $$\partial_{K}\psi(z,K)G=\left(\begin{array}{cc}
-\lambda |D|^{-1}G_{z}(x,u(x))\\
-\int_{M}G(x,u(x))dx
\end{array}
\right)
$$
So first by taking $G$ to be constant, we see that we can span the $\mathbb{R}$ component. For the other component we see that by taking $G(x,s(x))=<f(x),s(x)>$ then we have that the range of the first component is dense since $G_{z}$ can be any section of the spinor bundle and the operator $|D|^{-1}$ maps $C^{3}$ to a dense subspace. Thus we have the surjectivity. Therefore by the transversality theorem, $0$ is a regular point of $\psi(.,K)$ for a generic $K$ and this is equivalent to saying that $E_{H+K}$ is Morse.
\end{proof}
Notice also that the perturbation $K$ can be chosen to be even if $H$ is even, by changing the space of sections to even sections.
\begin{lemma}
Assume that $E$ is Morse and satisfies (PS) in $[a,b]$, then for every $\varepsilon>0$ there exists a functional $E^{\varepsilon}$ such that :\\
i)$||E-E^{\varepsilon}||_{C^{2}}<\varepsilon$\\
ii) $E^{\varepsilon}$ satisfies PS in $[a-\varepsilon,b+\varepsilon]$\\
ii)$E^{\varepsilon}$ has the same critical points as $E$ with the same connections (number of connecting orbits).
\end{lemma}
The proof of this result is very similar to the one in \cite{AB2} for that it will be omitted. 

\section{Computation of the Different Homologies}
Given $R>0$ we set $\rho_{R}$ to be a smooth function such that $\rho_{R}(s)=1$ for $s\in[0,R]$ and $\rho_{R}(s)=0$ for $s>R+1$.
\subsection{The regular case}

First let us compute this homology for the linear case, that is $H_{*}^{[-K,K]}$. In this case we know that the critical points are the couples of eigenvalues, eigenfunctions. And for each eigenvalue $\lambda_{k}$ there is a circle of eigenfunctions $e^{i\theta}\psi_{k}$. By a symmetry breaking argument we can break the circle to a min and a max as we did with in the construction of the homology in section 4. Therefore each critical circle of index $i_{0}$ will be broken to two critical points of index $i_{0}$ and $i_{0}+1$. Also notice that the critical circles have even index hence we have one generator for each index in the chain complex, that is $$C_{k}^{[-K,K]}=\mathbb{Z}_{2},$$ for all $k\in \mathbb{Z}$ such that $\lambda_{k}\in[-K,K]$. \\
Now let us compute the $\partial$ operator. Notice first that by construction the $\partial =0$ from odd index to even index. So it remains to compute it from even index to odd index, that is from min to max.\\ 
Let $(\lambda(t),u(t))$ be a flow line such that $$(\lambda(+\infty),u(+\infty))=(\lambda_{k},u_{k})$$ and $$(\lambda(-\infty),u(-\infty))=(\lambda_{k+1},u_{k+1}),$$
so if we write $u(t)$ in the basis $\varphi_{i}$ we get, if $u(t)=\sum_{i\in \mathbb{Z}}a_{i}(t)\varphi_{i}$, 
$$a_{i}'=(\lambda_{i}-\lambda(t))a_{i}.$$
Therefore we can write $a_{i}(t)=a_{i}(0)\exp(\int_{0}^{t}(\lambda_{i}-\lambda(s))ds)$. Using the convergence of $\lambda(t)$ at infinity we get that $a_{i}=0$ for $i\not \in \{k,k+1\}$ and we have (transversally to the $S^{1}$ action) exactly one flow line from the generator to the generator of the next index. Therefore the flow is in fact a finite dimensional one. As in the finite dimensional case we get then that $\partial $ is an isomorphism from even to odd.
Therefore, one gets $H_{*}=0$.\\

\begin{figure}[H]
\centering
\includegraphics[scale=.4]{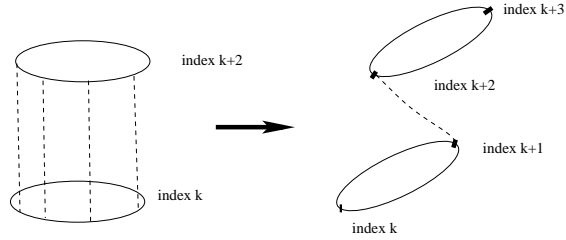}
\caption{Flow lines after Breaking the $S^{1}$ action}
\end{figure}

Using the stability result now of Section 5, we have that for every $H$ satisfying $(H1)$ and $(H2)$ we have that $H_{*}(H)=H_{*}(H_{0})=0$.
\subsection{The equivariant Case}

Let us consider again the linear case. The chain complex in this case is generated by a $\mathbb{Z}_{2}$ for each even index and the boundary operator is zero since we have a gap of 2 in the indices. Thus we have $H_{*}^{S^{1}}(H_{0})=\mathbb{Z}_{2}$ if $*$ is even and $H_{*}^{S^{1}}(H_{0})=0$ otherwise.\\
To conclude for the general case, notice that the deformations that we construct in the previous case preserves the $S^{1}$ action if we start with an $S^{1}$-equivariant $H$. Therefore $H_{*}^{S^{1}}(H)=\mathbb{Z}_{2}$.\\

A Similar computation in the case of the $\mathbb{Z}_{2}$ yields that $H_{*}^{\mathbb{Z}_{2}}(H)=\mathbb{Z}_{2}$.
\section{Applications}

\subsection{The sub-critical case}
If we consider the problem 
\begin{equation}
Du=\lambda h(x,u) \label{prob 1}
\end{equation} 
where $h$ has a potential $H$ satisfying $(H1)$ and $(H2)$ then
\begin{theorem}
The problem (\ref{prob 1}) has at least one solution. 
\end{theorem}
\begin{proof}
We consider $a<b$ so that the set $[a<E_{H_{0}}<b]$ contains only one critical point, that we will take for example to be the first eigenvalue-eigenfunction, $(\lambda_{1},u_{1})$. Now we have that $H_{0}^[a,b]\not = 0$. But looking back at the proof of the stability result by the uniform bounds that we find along the non-autonomous flow, one can show that there exists $M>0$ such that $H_{0}^{[a,b]}\hookrightarrow H_{0}^{[a-M,b+M]}(H)$. Hence we have at least one critical point with energy between $a-M$ and $b+M$.
\end{proof}

\begin{theorem}
If $H$ is $S^{1}$-equivariant or even, then problem (\ref{prob 1}) has infinitely many solutions $z_{k}=(\lambda_{k},u_{k})$, with $E_{H}(z_{k})$ diverges to infinity as $k$ tends to infinity.
\end{theorem}
\begin{proof}
In fact here we distinguish two cases, either we have the property that boundedness on the index implies boundedness of the energy and thus the result follows directly from the computation of the equivariant homology in theorem 2.2. Moreover we have a sequence of critical points with energy going to infinity as the index goes to infinity. \\
In the other case the theorem holds since we have an unbounded sequence of critical points with the same index.
\end{proof}
\begin{corollary}
The problem $Du=f(x)|u|^{p-1}u$ with $f\in L^{\infty}(M)$ and positive, has infinitely many solutions with energy going to $+\infty$.
\end{corollary}
\subsection{The critical case}
Let us recall here a version of the concentration-compactness proved for the Dirac Operator in the critical case.\\
Consider the following critical problem 
\begin{equation} \tag{$P_{\eta}$} \label{Pb crit}
Du=\eta u + |u|^{\frac{2}{n-1}}u.
\end{equation}

Its relative energy functional is then $$E^{\eta}(u)=\frac{1}{2}\int_{M}(<Du,u>-\eta \int_{M}u^2)-\frac{n-1}{2n}\int_{M}|u|^{\frac{2n}{n-1}}.$$ Hence the following holds :
\begin{lemma}[Isobe \cite{Iso2}]
Let $u_{i}$ be a (PS) sequence for the functional $E^{\eta}$ then there exist $k\geq 0$ and sequences $a_{i}^{j}\longrightarrow a^{j} \in M$, for $1\leq j \leq k$ and a sequence of numbers $R_{i}^{j}$ converging to zero, a solution $u_{\infty}\in H^{\frac{1}{2}}(M)$ of problem (\ref{Pb crit}) and solutions $u^{j} \in H^{\frac{1}{2}}(\mathbb{R}^{n})$ of $Du=|u|^{\frac{2}{n-1}}u$ on $\mathbb{R}^{n}$, such that up to subsequence, we have $$ u_{i}=u_{\infty}+\sum_{j=1}^{k}\omega_{i}^{j}+o(1) \text{ in } H^{\frac{1}{2}}(M),$$
where $$\omega_{i}^{j}=(R_{i}^{j})^{-\frac{n-1}{2}}\beta_{j}(x)(\rho_{i,j}^{-1})^{\ast}(u^{j})$$ and $$\rho_{i,j}(x)=\exp_{a^{j}}(R_{i}^{j}x),$$ here $\beta_{k}$ is a non negative function equals to $1$ in $B_{1}(a^j)$ and zero outside $B_{2}(a^j)$. Moreover, we have $$E^{\eta}(u_{i})=E^{\eta}(u_{\infty})+\sum_{j=0}^{k}E(u^{j}) +o(1).$$
\end{lemma}
If we go back to the previous proofs, leading us to compute the homology, we see that the main assumptions that are needed are just for the sake of the (PS) condition to hold and for the form of the non-linearity to get compactness in section 4. Notice also that for the stability we needed an adequate isolating neighborhood for the the non-autonomous flow to be construct.\\ 
We define $$Y(D,g_{0},M)=\inf_{g\in[g_{0}]}(\lambda_{1}(g)(D))$$ where $[g_{0}]$ is the set of all the metrics conformal to $g_{0}$ with constant volume $1$. This constant is equivalent to the Yamabe constant in the case of the conformal Laplacian $L_{g}$. In fact it coincides exactly with the Yamabe constant in that case. \\
Now we can easily prove the following result (as in  \cite{Rau})
\begin{theorem}
If $Y(D,g_{0},M)<Y(D,S^{n})=c(n),$ then problem ($P_{\eta}$) with $\eta=0$ has at least one solution.
\end{theorem}
\begin{proof}
Since $Y(D,g_{0},M)<Y(D,S^{n})$, there exists a metric $g_{1}$ with volume one, such that $\lambda_{1}(g_{1})< Y(D,S^{n})=c(n)$. Now using Lemma 8.4, we can see that $E_{|s|^{p+1}|}$ for $p=\frac{n+1}{n-1}$ satisfies the (PS) condition in the interval $[0,c(n))$. Now we want to show that $H_{*}^{[0,c(n)),S^{1}}(|s|^{p})=H_{*}^{[0,c(n)),S^{1}}(H_{0})$.\\
For the sake of contradiction, let us assume that $E_{|s|^{p+1}}$ has no critical point in $[0,c(n))$ and define the functional $E_{t}=\eta(t)E_{|s|^{p+1}}+(1-\eta(t))E_{H_{0}}$. Then we claim that the set $\tilde{N}$ defined by $$\tilde{N}=\{ (t,z)\in [-\epsilon,1+\epsilon]\times \mathcal{H}; 0\leq E_{t}(z)< c(n)\}$$
is an isolating neighborhood in the sense of Conley, for the flow $\frac{\partial z}{\partial t}=\nabla E_{t}(z)$.
Indeed if $z(t)$ is a flow-line, that stays in $\tilde{N}$ for all $t$ then when $t\longrightarrow -\infty$ $z(t)$ converges to a critical point of $E_{|s|^{p+1}}$ in the energy level $[0,c(n))$ which is impossible.
Thus the same procedure can be carried on as in section 7, to show that $H_{*}^{[0,c(n)),S^{1}}(|s|^{p})=H_{*}^{[0,c(n)),S^{1}}(H_{0})$. Now to finish the proof, it is enough to notice that for the metric $g_{1}$, the critical point corresponding to the first eigenvalue $\lambda_{1}$ is in the energy level $[0,c_{n})$ therefore $H_{*}^{[0,c(n))}(H_{0})\not = 0$ for $*=0$. Hence $H_{0}^{[0,c(n)),S^{1}}(|s|^{p+1})\not =0$ and we do have a critical point.
\end{proof}
Remark that what was helpful in the previous proof is the fact that the problem is conformally invariant. But if we consider the Brezis-Nirenberg type problem that is a problem with $\eta \not =0$, the invariance is broken. Despite this fact, we get the following result :
\begin{theorem}
If $\lambda_{1}>\eta>\lambda_{1}-c(n)$, then problem (\ref{Pb crit}) has at least one solution.
\end{theorem}
The proof here follows the same principle as above, since (PS) holds in $[0,c(n))$ hence, we apply our result to the operator $D-\eta$ instead of $D$. That is we shift the spectrum by $\eta$.\\
In fact the multiplicity result in this case is tied to a better understanding of the violation of the (PS) condition. For instance if it was discrete then we can try to isolate higher eigenvalues in between to gain compactness. This needs a classification of the solutions in the Euclidean case.
\subsection{Some remarks on the Yamabe problem}
In this case we will deal with the classical Yamabe problem on a compact manifold and try to see what one would get if we apply the previous method to the conformal Laplacian instead. One needs to be careful since we do not have any information from our construction on the sign of the solutions. That is, the solutions that we get might change sign.\\
Consider the following problem : $$L_{g}u=|u|^{p-1}u \label{Yam}$$ where $L_{g}=-\Delta_{g}+a_{n}R_{g}$ and $p=\frac{n+2}{n-2}$. It is known that the (PS) condition for the associated functional is violated only at given levels $b_{k}=(2\pi)^{2}\omega_{n}^{\frac{-2}{n}}k^{\frac{2}{n}}$ and away from that (PS) holds. Here, we have a better understanding on the loss of compactness compared to the Dirac operator.\\
Recall also that $$Y(M,g)=\inf_{u\in H^{1}(M);||u||_{L^{\frac{2n}{n-2}}}=1} \int |\nabla u|^{2}+a_{n}Ru^{2}$$
It is easy to see that if $\lambda_{1}(L_{g},M)$ denote the first eigenvalue of $L_{g}$, then $$Y(M,g)=\inf_{\tilde{g}\in [g]} \lambda_{1}(L_{\tilde{g}})Vol(\tilde{g})^{\frac{-1}{n}},$$ where $[g]$ is the conformal class of $g$. Hence if we fix the volume of our metrics to be always one, the following holds. 
\begin{theorem}
The changing sign Yamabe problem on a compact manifold has at least one solution.
\end{theorem}

\begin{proof}
The first case to be considered is when $Y(M,g)<b_{1}$. In this case we can always do a conformal change to the metric to get $\lambda_{1}(L_{g},M)<b_{1}$. This is done, we have in that case that $H^{[0,b_{1})\mathbb{Z}_{2}}_{*}(L_{g})\not = 0$ again as in Theorem 8.5 we get that $H^{[0,b_{1}),\mathbb{Z}_{2}}_{*}(L_{g},|u|^{p+1},)\not = 0$, therefore we have at least one solution in that energy range.\\
The other case is when $Y(M,g)=b_{1}=Y(S^{n},g_{0})$. So here, either $\lambda_{1}(L_{g},M)=Y(M,g)$ and the problem is solved since the inf is achieved by the eigenfunction. Or $\lambda_{1}(L_{g},M)>Y(M,g)$ and we can assume that up to a conformal change we have $\lambda_{1}(L_{g},M)<b_{2}$. Hence, since (PS) holds in $(b_{1},b_{2})$, the same reasoning can be done as in the previous case to get $H^{(b_{1},b_{2}),\mathbb{Z}_{2}}_{*}(L_{g},|u|^{p+1})\not = 0$. Therefore, we have at least one solution in that range.
\end{proof}
One has to notice that in fact the number of solutions is at least the number of eigenvalues that differs from the constants $b_{k}$. But recall the Weyl's asymptotic for the eigenvalues, that is $\lambda_{k}\thicksim (2\pi)^{2}\omega_{n}^{\frac{-2}{n}}k^{\frac{2}{n}}=b_{k}$. So unless we have exact equality in the asymptotic formula, we have infinitely many solutions. In particular if we have a repeated eigenvalue we have existence of two solutions in the previous theorem.\\

One also have the following result for the case of an open bounded set $\Omega$ of $R^{n}$. That is, if we consider the problem 
\begin{equation} \label{Crit op}
\left\{ \begin{array}{lll}
-\Delta u&=& |u|^{\frac{4}{n-2}}u \text{ in } \Omega\\
u&=&0 \text{ in } \partial \Omega
\end{array}
\right.
\end{equation}
\begin{theorem}
Let $\Omega_{\varepsilon}$ be a family of shrinking annuli with volume one, in $\mathbb{R}^{n}$, then there exists a solution for the changing sign Yamabe problem such that $E(u_{\epsilon})$ goes to infinity as $\varepsilon$ converges to zero.
\end{theorem}
\begin{proof}
This follows from the fact that the first eigenvalue of the Laplacian in this case blows-up when $\varepsilon$ converges to zero. Hence for $\varepsilon$ small enough, there exists $k_{\varepsilon}$ such that $b_{k_{\varepsilon}}<\lambda_{1}<b_{k_{\varepsilon}+1}$. Again using the same reasoning as before we get the desired result.
\end{proof}
In a first version of this paper, in Theorem 2.2 we made the assumption that the boundedness of the index implies the boundedness of the energy. For instance, this was proved for the case of the Laplacian operator as in \cite{BL} and for the case of systems with relative index as in \cite{An2}. This assumption was removed due to a change in the proof suggested by the referee.\\
 
\textbf{Conjecture: } We conjecture that in fact solutions to $Du=|u|^{p-1}u$ in $\mathbb{R}^{n}$ with finite relative Morse index should vanish for $p<\frac{n+1}{n-1}$. \\
 
\textbf{Acknowledgement:} I want to thank Dr. V.Martino for having carefully
proofread the manuscript. I also want to thank Professor A. Abbondandolo for his help and useful suggestions about different parts of the manuscript. I want to thank the referee for his suggestion on using the adiabatic argument method in the proof of the stability and his/her remarks that helped dropping one of the assumptions in the main result.

\end{document}